\pdfoutput=1
\RequirePackage{ifpdf}
\ifpdf 
\documentclass[pdftex]{sigma}
\else
\documentclass{sigma}
\fi

\usepackage{bm}
\usepackage[enableskew]{youngtab}
\usepackage{ytableau}

\numberwithin{equation}{section}

\newtheorem{Theorem}{Theorem}[section]
\newtheorem{Corollary}[Theorem]{Corollary}
\newtheorem{Lemma}[Theorem]{Lemma}
\newtheorem{Proposition}[Theorem]{Proposition}
\newtheorem{Conjecture}[Theorem]{Conjecture}
{ \theoremstyle{definition}
\newtheorem{Remark}[Theorem]{Remark}
\newtheorem{Example}[Theorem]{Example}
}

\newcommand{\bR}{\mathbb{R}}
\newcommand{\bQ}{\mathbb{Q}}
\newcommand{\mf}[1]{\mathfrak{#1}}
\newcommand{\mcal}[1]{\mathcal{#1}}

\begin{document}

\newcommand{\arXivNumber}{1803.01121}

\renewcommand{\thefootnote}{}

\renewcommand{\PaperNumber}{053}

\FirstPageHeading

\ShortArticleName{A Spin Analogue of Kerov Polynomials}

\ArticleName{A Spin Analogue of Kerov Polynomials\footnote{This paper is a~contribution to the Special Issue on the Representation Theory of the Symmetric Groups and Related Topics. The full collection is available at \href{https://www.emis.de/journals/SIGMA/symmetric-groups-2018.html}{https://www.emis.de/journals/SIGMA/symmetric-groups-2018.html}}}

\Author{Sho MATSUMOTO}

\AuthorNameForHeading{S.~Matsumoto}

\Address{Graduate School of Science and Engineering, Kagoshima University,\\
Kagoshima 890-0065, Japan}
\Email{\href{mailto:shom@sci.kagoshima-u.ac.jp}{shom@sci.kagoshima-u.ac.jp}}

\ArticleDates{Received March 13, 2018, in final form May 29, 2018; Published online June 02, 2018}

\Abstract{Kerov polynomials describe normalized irreducible characters of the symmetric groups in terms of the free cumulants associated with Young diagrams. We suggest well-suited counterparts of the Kerov polynomials in spin (or projective) representation settings. We show that spin analogues of irreducible characters are polynomials in even free cumulants associated with double diagrams of strict partitions. Moreover, we present a conjecture for the positivity of their coefficients.}


\Classification{05E10; 20C30; 05E05}

\renewcommand{\thefootnote}{\arabic{footnote}}
\setcounter{footnote}{0}

\section{Introduction}

\subsection{Characters of symmetric groups}

Irreducible representations of symmetric groups $S_n$ are indexed by partitions $\lambda$ of~$n$, or equivalently by Young diagrams of size~$n$. The corresponding character $\chi^\lambda$ takes values at conjugacy classes $C_\nu$ in $S_n$, which are also indexed by partitions~$\nu$ of~$n$. The character values $\chi^\lambda_\nu=\chi^\lambda(C_\nu)$ have been studied for a long time by innumerable researchers, see, e.g.,~\cite{Sagan}. Recently, for se\-veral problems in the asymptotic representation theory, we often deal with the $\chi^\lambda_\nu$, fixing $\nu$ and letting~$\lambda$ vary. More precisely, for a partition~$\nu$ of~$k$ and a partition $\lambda$ of~$n$ with $k \le n$, we define
\begin{gather} \label{eq:normalized_ch}
\mathrm{Ch}_\nu(\lambda)= n(n-1) \cdots (n-k+1) \frac{\chi^\lambda_{\nu \cup (1^{n-k})}}{f^\lambda},
\end{gather}
where $f^\lambda :=\chi^\lambda_{(1^{|\lambda|})}$ denotes the dimension of the irreducible representation of the symmetric group $S_{|\lambda|}$ associated with~$\lambda$. We set $\mathrm{Ch}_\nu(\lambda)=0$ whenever $k>n$.
Then the $\mathrm{Ch}_\nu$ are functions on all Young diagrams. An approach from this point of view is sometimes referred to as \emph{dual approach} or \emph{dual combinatorics} to characters of symmetric groups \cite{DolegaFeraySniady2014, FeraySniady2011b}.

\subsection{Kerov polynomials}

Biane and Kerov discovered that free cumulants play an important role in the asymptotic representation theory of symmetric groups \cite{Biane1998, Biane2003}. Here the free cumulants $R_j(\lambda)$ $(j=1,2,\dots)$ of the transition measure of a Young diagram $\lambda$ are sequences of real numbers defined in the framework of the free probability theory. (Note that $R_1 \equiv 0$.) Then the normalized character $\mathrm{Ch}_{k}:=\mathrm{Ch}_{(k)}$ for one-row partition $\nu=(k)$ can be expressed as a \emph{polynomial} in $R_2,\dots, R_{k+1}$ with integer coefficients.

\begin{Theorem}[Kerov's character formula \cite{Biane2003}] For each $k=1,2,3,\dots$, there exists a~polynomial~$K_k$ in $k$ variables with integer coefficients, such that
\begin{gather*}
\mathrm{Ch}_{k}(\lambda)= K_k(R_2(\lambda),R_3(\lambda), \dots, R_{k+1}(\lambda))
\end{gather*}
for all partitions $\lambda$. Furthermore, $K_k$ is of the form
\begin{gather*}
K_k(R_2,\dots,R_{k+1})= R_{k+1}+ (\text{a polynomial in $R_2,\dots, R_{k-1}$}).
\end{gather*}
\end{Theorem}

\begin{Example}[\cite{Biane2003}] \label{ex:KerovPoly}
\begin{gather*}
K_1 =R_2, \\
K_2 =R_3, \\
K_3 =R_4+R_2, \\
K_4 =R_5+5R_3, \\
K_5 =R_6+15R_4+ 5 R_2^2 + 8R_2, \\
K_7 = R_8 + 70 R_6+ 84 R_4 R_2+ 56 R_3^2
+14 R_2^3 + 469 R_4+ 224 R_2^2+ 180 R_2, \\
K_9= R_{10}+210 R_8+300 R_6 R_2 +480 R_5 R_3+270 R_4^2 + 360 R_3^2 R_2 +270 R_4 R_2^2+ 30 R_2^4 \\
\hphantom{K_9=}{} + 5985 R_6+ 10548 R_4 R_2 +6714 R_3^2+2400 R_2^3 +26060 R_4 +
14580 R_2^2 +8064 R_2.
\end{gather*}
Here we omit $K_6$ and $K_8$ because we will compare odd-numbered polynomials with Example~\ref{ex:SpinKerov} below. We can see a complete list of $K_k$ up to $k=11$ in~\cite{Biane2003}.
\end{Example}

The polynomials $K_k$ in the theorem are called \emph{Kerov polynomials}. We can observe that all coefficients in the above examples are \emph{nonnegative} integers. This surprising phenomenon is very nontrivial and had been called the Kerov (positivity) conjecture. It was finally proved by F\'{e}ray~\cite{Feray2009}. The authors in~\cite{DolegaFeraySniady2010} obtained clearer combinatorial interpretations.

\begin{Theorem}[F\'{e}ray \cite{Feray2009}] \label{thm:PositivityFeray} All coefficients of $K_k$ are nonnegative integers.
\end{Theorem}

Recently, a generalization of Kerov polynomials involving Jack polynomials has been actively studied in \cite{DolegaFeray2016, Lassalle, Sniady2016}. In this paper, we work on research in another direction. Our aim here is to present an answer to the following question: \emph{What is the counterpart of the Kerov polynomials in the spin representation setting?}

\subsection{Spin representations}

Spin (or projective) representation theory of symmetric groups was introduced by Schur. Currently, it can be understood via a double covering of the symmetric group or via the Hecke--Clifford algebra, see, e.g.,~\cite{WanWang2011}. In the spin case, the corresponding character values $X^\lambda_\nu$ are indexed by strict partitions $\lambda$ and odd partitions~$\nu$. For a fixed odd partition~$\nu$, we can define a spin version~$\mathfrak{p}_\nu$ of~\eqref{eq:normalized_ch}, which is a function on the set of all strict partitions. The function~$\mathfrak{p}_\nu$ was first introduced by Ivanov~\cite{Ivanov2004, Ivanov2006} (with notation $p_\nu^\#$) and developed in author's recent work~\cite{M}.

\subsection{Results}

To define a spin analogue of Kerov polynomials, we need to find a spin analogue of free cumu\-lants~$R_j(\lambda)$. Furthermore, it is natural to expect that the same statement of Theorem~\ref{thm:PositivityFeray} holds in the spin setting. The aim in this paper is to establish a spin analogue of Kerov polynomials by using a~spin analogue $\mathfrak{R}_{2j}(\lambda)$ of $R_j$, which is defined as the \emph{half} of \emph{even} free cumulants of the \emph{double diagram} of a strict partition $\lambda$. Indeed, we obtain the following theorem for $\mathfrak{p}_{2k-1}$ with one-row odd partition $\nu=(2k-1)$.

\begin{Theorem} \label{thm:MainThm} For each $k=1,2,3,\dots$, there exists a polynomial $K_{2k-1}^{\mathrm{spin}}$ in $k$ variables with rational coefficients, such that
\begin{gather*}
\mathfrak{p}_{2k-1}(\lambda)= K_{2k-1}^{\mathrm{spin}} (\mathfrak{R}_2(\lambda), \mathfrak{R}_4(\lambda),\dots, \mathfrak{R}_{2k}(\lambda))
\end{gather*}
for all strict partitions $\lambda$. Furthermore, $K_{2k-1}^{\mathrm{spin}}$ is of the form
\begin{gather*}
K_{2k-1}^{\mathrm{spin}}= \mathfrak{R}_{2k} + (\text{a polynomial in $\mathfrak{R}_2, \mathfrak{R}_4,\dots, \mathfrak{R}_{2k-2}$ of degree at most $2k-2$}).
\end{gather*}
\end{Theorem}

Here we only show that the coefficients in $K_{2k-1}^{\mathrm{spin}}$ are rational numbers. The degree of $\mathfrak{R}_{2k}$ is regarded as $\deg \mathfrak{R}_{2k}=2k-1$. We call the polynomial $K_{2k-1}^{\mathrm{spin}}$ a \emph{spin Kerov polynomial}.

\begin{Example} \label{ex:SpinKerov}
\begin{gather*}
K_1^{\mathrm{spin}}= \mathfrak{R}_2, \\
K_3^{\mathrm{spin}}= \mathfrak{R}_4+\mathfrak{R}_2, \\
K_5^{\mathrm{spin}}= \mathfrak{R}_6+ 15\mathfrak{R}_4 + 10 \mathfrak{R}_2^2 +8 \mathfrak{R}_2, \\
K_7^{\mathrm{spin}}= \mathfrak{R}_8+70\mathfrak{R}_6+ 168\mathfrak{R}_4 \mathfrak{R}_2+ 56 \mathfrak{R}_2^3+ 469 \mathfrak{R}_4 +560 \mathfrak{R}_2^2 +180 \mathfrak{R}_2, \\
K_9^{\mathrm{spin}}= \mathfrak{R}_{10} +210 \mathfrak{R}_{8}+600 \mathfrak{R}_{6} \mathfrak{R}_{2} + 540 \mathfrak{R}_{4}^2+1080 \mathfrak{R}_{4} \mathfrak{R}_{2}^2 +240 \mathfrak{R}_{2}^4 \\
\hphantom{K_9^{\mathrm{spin}}=}{} +5985 \mathfrak{R}_{6} +23016 \mathfrak{R}_{4} \mathfrak{R}_{2}+9120 \mathfrak{R}_{2}^3+26060 \mathfrak{R}_{4} + 41628\mathfrak{R}_{2}^2 + 8064 \mathfrak{R}_{2}.
\end{gather*}
\end{Example}

It seems that formulas in Example~\ref{ex:SpinKerov} resemble those in Example~\ref{ex:KerovPoly}. The proof of Theorem~\ref{thm:MainThm} is accomplished by observing relations among some collections of generators in a symmetric function algebra~$\Gamma$. We obtain formulas in Example~\ref{ex:SpinKerov} as the by-products. Not only Theorem~\ref{thm:MainThm} but also Example~\ref{ex:SpinKerov} is an important result in this paper because the latter one indicates a spin analogue of Theorem~\ref{thm:PositivityFeray} to us.

\begin{Conjecture}[spin Kerov conjecture] \label{SpinKerovConj} All coefficients in spin Kerov polynomials are nonnegative integers.
\end{Conjecture}

\subsection{Discussion and outline}

Comparing Example~\ref{ex:SpinKerov} with Example~\ref{ex:KerovPoly}, we can find some interesting coincidences for coefficients. For example, the coefficient of linear terms $\mathfrak{R}_{2j}$ in $K_{2k-1}^{\mathrm{spin}}$ likely coincides with that of~$R_{2j}$ in~$K_{2k-1}$. Furthermore, if $2s_2+4s_4+ \cdots =2k-2$, the coefficient of $\mathfrak{R}_2^{s_2} \mathfrak{R}_4^{s_4} \cdots $ in $K_{2k-1}^{\mathrm{spin}}$ likely coincide with $2^{s_2+s_4+\cdots-1}$ times the coefficient of $R_2^{s_2} R_4^{s_4} \cdots$ in $K_{2k-1}$. We do not discuss their coincidence or combinatorial interpretations here. We leave them in future work.

We give the proof of Theorem \ref{thm:MainThm} and how to derive Example~\ref{ex:SpinKerov} with precise definitions in Section~\ref{sec:MainPart} after the review of preliminary facts for free cumulants Section~\ref{sec:Preliminary}. As mentioned, the ordinary Kerov conjecture was first proved by F\'{e}ray~\cite{Feray2009}. After that, the authors of~\cite{DolegaFeraySniady2010} gave alternative proof by using Stanley--F\'{e}ray polynomials. The Stanley--F\'{e}ray polynomials were analyzed by using Young symmetrizers \cite{FeraySniady2011a}. In order to attack Conjecture~\ref{SpinKerovConj}, it is a~natural practice to study projective Young symmetrizers. The projective Young symmetrizer was studied in \cite{Jones1998, Nazarov1997}. Unfortunately it was quite complicated, so it seems difficult to apply it to our problem.

In Section~\ref{sec:another_double}, we try another choice for free cumulants. The another free cumulants $\mathtt{R}_{2k}(\lambda)$ comes from a symmetrized double diagram $\mathtt{D}(\lambda)$ of $\lambda$. The symmetrized diagrams $\mathtt{D}(\lambda)$ are useful for the study of asymptotics of Plancherel measures in the spin setting (see \cite{Ivanov2006} and \cite[Chapter~4]{DarioThesis}). Since odd-numbered quantities $\mathtt{R}_{2k-1}$ vanish by virtue of the symmetry of $\mathtt{D}(\lambda)$, they are simpler than free cumulants of $D(\lambda)$ in a sense. Even if we replace $\mathfrak{R}_{2k}$ with $\mathtt{R}_{2k}$, we can show the existence of ``spin Kerov polynomials'' easily. However, the corresponding polynomial for $\mathfrak{p}_{2k-1}$ with $k=2$ has a \emph{non-integer} coefficient. In this sense, we view the $\mathtt{R}_{2k}$ as being unsuitable for spin Kerov polynomials. We thus believe that the~$\mathfrak{R}_{2k}$ are the most appropriate choice in the spin setting.

\section{Preliminary} \label{sec:Preliminary}

In the present section, we review the transition measures of a partition according to \cite{Hora, IvanovOlshanski}.

\subsection{Partitions}

A \emph{partition} $\lambda=(\lambda_1,\lambda_2,\dots)$ is a weakly decreasing sequence of nonnegative integers satisfying $|\lambda|=\sum\limits_{i \ge 1} \lambda_i<\infty$. When $|\lambda|=n$, we say $\lambda$ to be a partition of $n$ and sometimes write as $\lambda \vdash n$. The number of nonzero $\lambda_i$ is called the length of $\lambda$ and written as $\ell(\lambda)$.

We usually identity a partition with its Young diagram. We denote by $\lambda'=(\lambda_1',\lambda_2',\dots)$ the conjugate partition of $\lambda$, i.e., the Young diagram of $\lambda'$ is the transpose of that of $\lambda$ with respect to the diagonal line. Define $d=d(\lambda)$ by the number of boxes on the diagonal line in the Young diagram of $\lambda$. The \emph{modified Frobenius notation}
\begin{gather*}
[a_1,\dots,a_d \, | \, b_1,\dots, b_d]
\end{gather*}
of $\lambda$ is determined by
\begin{gather*}
a_i = \lambda_i-i+\frac{1}{2}, \qquad b_i = \lambda_i'-i+\frac{1}{2}, \qquad i=1,2,\dots, d.
\end{gather*}
Note that $\sum\limits_{i=1}^d (a_i+b_i)=|\lambda|$. For example, the modified Frobenius notation of $\lambda=(5,4,4,1,1) \vdash 15$ is $\big[4+\frac{1}{2}, 2+\frac{1}{2}, 1+\frac{1}{2} \, | \, 4+\frac{1}{2}, 1+\frac{1}{2}, 0+\frac{1}{2}\big]$.

\subsection{Free cumulants}

Let $\mu$ be a probability measure on $\mathbb{R}$ with a compact support. Define the Cauchy transform of~$\mu$ by
\begin{gather*}
G_\mu(z)= \int_\bR \frac{1}{z-x} \mu (dx) = \sum_{k=0}^\infty \frac{M_k[\mu]}{z^{k+1}},\qquad z \in \mathbb{C} \setminus \mathbb{R}, \quad |z| \gg 1,
\end{gather*}
where the $M_k[\mu]$ are moments of $\mu$:
\begin{gather*}
M_k[\mu]= \int_\bR x^k \mu(dx).
\end{gather*}
\emph{Free cumulants} $R_k[\mu]$ $(k=1,2,\dots)$ of $\mu$ are defined via the famous free cumulant-moment formula, see, e.g., \cite[Chapter~1]{Hora}. We do not need the explicit definition here. We only use the relation
\begin{gather} \label{eq:CM_formula}
R_k = M_k+\sum_{\substack{\nu=(\nu_1,\nu_2,\dots) \vdash k \\ \nu_1<k}}c_\nu M_{\nu_1}M_{\nu_2} \cdots
\end{gather}
with some integer coefficients $c_\nu$.

\subsection{Kerov's transition measures} \label{subsec:transition_measures}

We draw the Young diagram of a partition $\lambda$ in Russian style, see, e.g., \cite[Fig.~2.1]{Hora}. Let
\begin{gather*}
x_1< y_1< \cdots <x_{r-1}<y_{r-1}<x_{r}
\end{gather*}
be the corresponding local minimas and maximas, which are integers by definition. We call these numbers the \emph{Kerov interlacing coordinates} of $\lambda$. It is known that they satisfy the relation \cite[Lemma~2.1]{Hora}
\begin{gather} \label{eq:interlacing_relation}
\sum_{i=1}^r x_r = \sum_{j=1}^{r-1} y_{j}.
\end{gather}

For each partition $\lambda$, we define a probability measure $\mathfrak{m}_\lambda$ on $\mathbb{R}$ via the Cauchy transform
\begin{gather*}
G_{\mathfrak{m}_\lambda}(z)=\frac{\prod\limits_{i=1}^{r-1} (z-y_i)} {\prod\limits_{i=1}^r (z-x_i)}.
\end{gather*}
This probability measure, called \emph{Kerov's transition measure} of $\lambda$, is supported by the set $\{x_1,\dots,x_r\}$. Denote by $R_k(\lambda)$ $(k=1,2,\dots)$ the free cumulant of the measure $\mathfrak{m}_\lambda$: $R_k (\lambda)= R_k[\mathfrak{m}_\lambda]$. It has the expression \cite[Theorem~2]{Biane2003}
\begin{gather} \label{eq:Biane_FC}
R_k(\lambda)= -\frac{1}{k-1} \big[z^{-1}\big] G_{\mathfrak{m}_\lambda}(z)^{-(k-1)}
\end{gather}
for $k \ge 2$.

\subsection[Super power-sums $p_k^{\mathrm{super}}$]{Super power-sums $\boldsymbol{p_k^{\mathrm{super}}}$}\label{subsec:super_symmetric_poly}

For each partition $\lambda$ with modified Frobenius expression $[a_1,\dots,a_d \, | \, b_1,\dots, b_d]$
and for $k=1,2,\dots$, we define
\begin{gather*}
p_k^{\mathrm{super}}(\lambda)=\sum_{i=1}^d (a_i^k +(-1)^{k-1} b_i^k).
\end{gather*}
These are also determined via
\begin{gather} \label{eq:phi_Frobenius}
\phi(z;\lambda) := \prod_{i=1}^d \frac{z+b_i}{z-a_i}
=\exp \left( \sum_{k =1}^\infty \frac{p_k^{\mathrm{super}}(\lambda)}{k}
\frac{1}{z^k}\right).
\end{gather}
Note that $p_1^{\mathrm{super}}(\lambda)=\sum\limits_{i=1}^d (a_i+b_i)=|\lambda|$.

\begin{Lemma}[{\cite[Proposition 2.1]{Hora}}]\label{lem:G_to_phi} For each partition $\lambda$, we have
\begin{gather*}
z G_{\mathfrak{m}_\lambda}(z) = \frac{\phi\big(z-\frac{1}{2};\lambda\big)}{\phi\big(z+\frac{1}{2};\lambda\big)}.
\end{gather*}
\end{Lemma}

\subsection{Rayleigh measures}\label{subsec:Rayleigh_measures}

For each partition $\lambda$ with Kerov's interlacing coordinate $x_1<y_1< \cdots<y_{r-1}<x_r$, we introduce an $\mathbb{R}$-valued measure $\tau_\lambda$ on $\mathbb{R}$ by
\begin{gather*}
\tau_\lambda = \sum_{i=1}^r \delta_{x_i} - \sum_{i=1}^{r-1} \delta_{y_i},
\end{gather*}
where $\delta_x$ is the Dirac measure at a point $x$. This is called the \emph{Rayleigh measure} of $\lambda$. The moments are clearly given by
\begin{gather*}
M_k[\tau_\lambda]= \sum_{i=1}^r x_i^k - \sum_{i=1}^{r-1} y_i^k
\end{gather*}
for $k=1,2,\dots$. Note that $M_1[\tau_\lambda]=0$ by \eqref{eq:interlacing_relation}.

\begin{Lemma}[{\cite[Proposition 2.4]{Hora}}] \label{lem:ordinary_Rayleigh_super} For $n \ge 2$,
\begin{gather*}
M_n[\tau_\lambda]=\sum_{j=0}^{\lfloor (n/2)-1 \rfloor} \binom{n}{2j+1} 2^{-2j} p_{n-2j-1}^{\mathrm{super}}(\lambda) = n p_{n-1}^{\mathrm{super}}(\lambda) +\cdots.
\end{gather*}
\end{Lemma}

\subsection{Relations between two measures}

Two measures $\mathfrak{m}_\lambda$ and $\tau_\lambda$ have the following relation, which is nothing but that the relation between complete symmetric functions and power-sum symmetric functions.

\begin{Lemma}[{\cite[Proposition 2.2]{Hora}}] \label{lem:transition_Rayleigh}
\begin{gather*}
1+\sum_{n=1}^\infty M_n [\mf{m}_\lambda] \frac{1}{z^n} = \exp \left( \sum_{k=1}^\infty \frac{M_k[\tau_\lambda]}{k} \frac{1}{z^k} \right).
\end{gather*}
Equivalently,
\begin{gather*}
M_n [\mf{m}_\lambda] = \sum_{\nu=(\nu_1,\nu_2,\dots) \vdash n} z_\nu^{-1} M_{\nu_1}[\tau_\lambda]M_{\nu_2}[\tau_\lambda] \cdots,
\end{gather*}
where $z_\nu= \prod\limits_{i \ge 1} i^{m_i} m_i!$ with $m_i=m_i(\nu)=|\{j \ge 1 \, | \, \nu_j=i\}|$.
\end{Lemma}

Combining above formulas, we can express free cumulants $R_k(\lambda)$ in terms of functions $p_j^{\mathrm{super}}(\lambda)$ in principle. This fact will be applied in the next section.

\section{Spin Kerov polynomials} \label{sec:MainPart}

Our goal of this section is to define the counterpart of the Kerov polynomial in the spin representation setting.

\subsection{Strict and odd partitions}

A partition $\lambda=(\lambda_1,\lambda_2,\dots)$ is said to be \emph{strict} if its nonzero parts $\lambda_i$ are distinct. Let $\mathcal{SP}$ be the set of all strict partitions. For a strict partition $\lambda =(\lambda_1>\lambda_2>\cdots>\lambda_l>0)$, we define the \emph{double} of $\lambda$ by
\begin{gather} \label{eq:double_Frobenius}
D(\lambda)=\big[\lambda_1+\tfrac{1}{2},\lambda_2+\tfrac{1}{2}, \dots, \lambda_l+\tfrac{1}{2} \, | \, \lambda_1-\tfrac{1}{2},\lambda_2-\tfrac{1}{2}, \dots, \lambda_l-\tfrac{1}{2}\big]
\end{gather}
in the modified Frobenius notation. Note that $|D(\lambda)|=2|\lambda|$.

\begin{Example}For a strict partition $\lambda=(5,4,2,1)$, we have three kinds of diagrams as below. On the left, the Young diagram (in English style) is a collection of left-justified rows of boxes where the $i$-th row has $\lambda_i$ boxes reading from top to bottom. On the middle, the shifted diagram is obtained from the Young diagram by shifting the $i$-th row $(i-1)$ boxes to the right, for each $i \ge 2$. Moreover, on the right, the Young diagram of $D(\lambda)$, or the double diagram of $\lambda$, is obtained by combining the shifted diagram with its reflection (the gray area) on the diagonal.
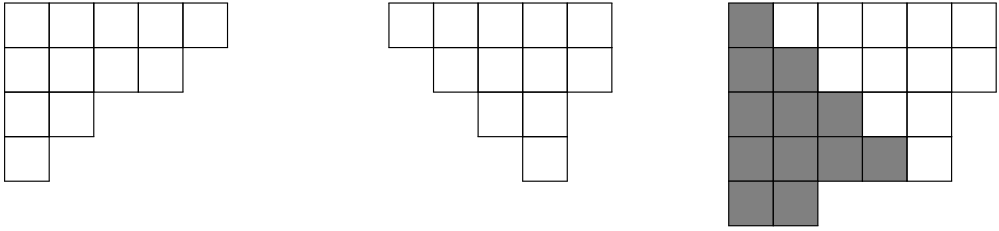
\begin{gather*}
\ydiagram{5,4,2,1} \qquad \qquad
\ydiagram{1+5,2+4,3+2,4+1} \qquad \qquad
\ydiagram{1+5,2+4,3+2,4+1}*[*(gray)]{1,2,3, 4,2}
\end{gather*}
\end{Example}

A (not necessary strict) partition $\rho=(\rho_1,\rho_2,\dots)$ is said to be \emph{odd} if all nonzero $\rho_i$ are odd. It is well known that the number of strict partitions of $n$ coincides with that of odd partitions of~$n$.

\subsection{Symmetric functions}

We review some symmetric functions according to Macdonald's book \cite[Chapter~III.8]{Mac}. Recall the power-sum symmetric function
\begin{gather*}
p_k(\bm{x}) =x_1^{k} +x_2^k+x_3^k+\cdots
\end{gather*}
in infinitely many variables $\bm{x}=(x_1,x_2,x_3,\dots)$. For an odd partition $\rho=(\rho_1,\dots,\rho_l)$, set
$p_\rho=p_{\rho_1} p_{\rho_2} \cdots p_{\rho_l}$.

Let $\Gamma$ be the $\mathbb{Q}$-algebra generated by odd power-sum symmetric functions $p_k$ $(k=1,3,5,\dots)$.
The degree on $\Gamma$ is naturally defined by
\begin{gather*}
\deg p_{k}=k, \qquad k=1,3,5,\dots.
\end{gather*}

For each $\lambda \in \mcal{SP}$, denote by $P_\lambda$ Schur's $P$-function, see the definition in \cite[Chapter~III.8]{Mac}. The family $\{P_\lambda \, | \, \lambda \in \mcal{SP}\}$ and $\{p_\rho \, | \, \text{$\rho$ are odd partitions}\}$ form a linear basis of $\Gamma$, respectively. The quantity $X^\lambda_\rho$, where $\lambda$ is a strict partition and $\rho$ is an odd partition with $|\lambda|=|\rho|$, is defined via the relation
\begin{gather*}
p_\rho= \sum_{\lambda: |\lambda|=|\rho|} X^\lambda_\rho P_\lambda.
\end{gather*}
Put $g^\lambda=X^{\lambda}_{(1^{|\lambda|})}$. These quantities $X^\lambda_\rho$ are integers and encode character values of irreducible spin representations of symmetric groups~\cite{HH}. In particular, a positive integer $g^\lambda$ times a power of $2$ coincides with the dimension of an irreducible spin representation of symmetric groups.

Each symmetric function $f$ in $\Gamma$ is regarded as a $\mathbb{Q}$-valued function on $\mcal{SP}$, by
\begin{gather*}
f(\lambda)=f(\lambda_1,\lambda_2,\dots,\lambda_l,0,0,\dots), \qquad \lambda=(\lambda_1>\lambda_2>\cdots >\lambda_l>0).
\end{gather*}
In particular, $p_k(\lambda)=\sum\limits_{i =1}^l \lambda_i^k$. For two functions $f,g \in \Gamma$, it holds $f(\lambda)=g(\lambda)$ for all $\lambda \in \mcal{SP}$ if and only if $f=g$ \cite[Proposition~6.2]{Ivanov2004}.

\subsection{Spin characters}

The spin (or projective) analogue of $\mathrm{Ch}_\nu$ given in \eqref{eq:normalized_ch} is defined as follows \cite{Ivanov2004, Ivanov2006, M}. For an odd partition $\rho$ of $k$ and strict partition $\lambda$ of $n$ with $k \le n$, we define
\begin{gather*}
\mathfrak{p}_\rho(\lambda)= n(n-1)\cdots (n-k+1) \frac{X^{\lambda}_{\rho \cup (1^{n-k})}}{g^\lambda},
\end{gather*}
where $\rho \cup \big(1^{n-k} \big)=(\rho_1,\rho_2,\dots,\rho_l,1,1,\dots,1) \vdash n$. Set $\mathfrak{p}_\rho(\lambda)=0$ for $k >n$. In this paper, we focus the spin characters $\mathfrak{p}_{2k-1}=\mathfrak{p}_{(2k-1)}$ for one-row odd partitions $\rho=(2k-1)$.

The collection $\{\mathfrak{p}_{2k-1} \, | \, k=1,2,3,\dots \}$ forms an algebraic basis of $\Gamma$ \cite[Proposition 6.4]{Ivanov2004}. More specifically, we have
\begin{gather} \label{eq:spinCh_p1}
\mathfrak{p}_{2k-1}= p_{2k-1}+ (\text{a polynomial in $p_1,p_3,\dots, p_{2k-3}$ of degree smaller than $2k-1$)}.
\end{gather}

\begin{Example} \label{ex:SpinCharacter}The functions $\mathfrak{p}_{2k-1}$ for $1 \le k \le 5$ are expanded in terms of $p_j$ as follows
\begin{gather*}
\mf{p}_{1}= p_1, \\
\mf{p}_{3}= p_3-3p_1^2+2p_1, \\
\mf{p}_{5}= p_5 -10 p_{3} p_1+\frac{55}{3} p_3 +\frac{50}{3} p_{1}^3 -50p_1^2 +24 p_1, \\
\mathfrak{p}_7= p_7-14 p_5 p_1-7 p_3^2+77 p_5 +98 p_3 p_1^2-\frac{1862}{3} p_3 p_1 -\frac{343}{3} p_1^4 \\
\hphantom{\mathfrak{p}_7=}{} +\frac{2128}{3} p_3+\frac{2744}{3} p_1^3 -1764 p_1^2+720 p_1, \\
\mathfrak{p}_9= p_9 -18 p_7 p_1 -18 p_5 p_3 +222 p_7 +162 p_5 p_1^2+ 162 p_3^2 p_1-2538 p_5 p_1 -1026 p_3^2-972 p_3 p_1^3 \\
\hphantom{\mathfrak{p}_9=}{} + \frac{37401}{5} p_5 +14094 p_3 p_1^2+ \frac{4374}{5} p_1^5
-52704 p_3 p_1 - 14580 p_1^4 \\
\hphantom{\mathfrak{p}_9=}{} + 47492 p_3+ 70632 p_1^3 -109584 p_1^2 +40320 p_1.
\end{gather*}
\end{Example}

These examples are obtained by using the formula given in \cite[Proposition 3.3]{Ivanov2006}: For odd $k=1,3,5,\dots$,
\begin{gather*}
\mathfrak{p}_{k}(\lambda)= \big[z^{-1}\big] \left(-\frac{1}{4k}\right) (2z-k)\prod_{j=1}^{k-1} (z-j) \cdot \frac{\Phi(z;\lambda)}{\Phi(z-k;\lambda)},
\end{gather*}
where $\Phi(z;\lambda)$ is defined by \eqref{eq:Phi_definition} or \eqref{eq:Phi_definition2} below. Here $\big[z^{-1}\big]Q(z)$ stands for the coefficient of $z^{-1}$ in the Laurent series expansion of $Q(z)$ at $z=\infty$.

\begin{Remark} The formulas in Example \ref{eq:spinCh_p1} are also obtained in the following way. First, we expand $\mathfrak{p}_\rho$ in terms of factorial Schur $P$-functions~$P^*_\lambda$.
Second, we expand each $P^*_\lambda$ in terms of (ordinary) Schur $P$-functions~$P_\nu$. Finally, we expand each $P_\nu$ in terms of odd power-sum symmetric functions~$p_\sigma$. See \cite[Example~3.3]{M} for details.
\end{Remark}

\subsection{Super symmetric polynomials}

Let $\lambda$ be a strict partition. Put
\begin{gather} \label{eq:Phi_definition}
\Phi(z;\lambda)= \prod_{i=1}^{\ell(\lambda)} \frac{z+\lambda_i}{z-\lambda_i}.
\end{gather}
It is easy to see that
\begin{gather} \label{eq:Phi_definition2}
\log \Phi(z;\lambda)= 2 \sum_{k=1}^\infty \frac{p_{2k-1}(\lambda)}{2k-1} \frac{1}{z^{2k-1}}.
\end{gather}
Recall functions $p_n^{\mathrm{super}}$ introduced in Section~\ref{subsec:super_symmetric_poly}. The following proposition is a key in the proof of Theorem \ref{thm:MainThm}.

\begin{Proposition} \label{prop:strict_super_p} For each strict partition $\lambda$ and $n=1,2,3,\dots$, we have
\begin{gather*}
p_n^{\mathrm{super}}(D(\lambda))= \sum_{j=0}^{\lfloor (n-1)/2 \rfloor} \binom{n}{2j+1} \frac{1}{2^{n-2j-2}} p_{2j+1}(\lambda).
\end{gather*}
In particular, for each $k=1,2,3,\dots$,
\begin{gather*}
p_{2k-1}^{\mathrm{super}}(D(\lambda))= 2p_{2k-1}(\lambda)+ \cdots \qquad \text{and} \qquad p_{2k}^{\mathrm{super}}(D(\lambda))= 2k p_{2k-1}(\lambda)+ \cdots,
\end{gather*}
where dots are linear combinations of $\{p_{2j-1}\}_{j=1,2,\dots,k-1}$ with $\mathbb{Q}_{\ge 0}$-coefficients. Therefore the family $\big\{p_{2k-1}^{\mathrm{super}}(D(\cdot))\big\}_{k=1,2,3,\dots}$ is an algebraic basis of $\Gamma$.
\end{Proposition}

\begin{proof}
Using \eqref{eq:phi_Frobenius} and \eqref{eq:double_Frobenius}, we have
\begin{gather} \label{eq:phi_to_Phi}
\phi(z; D(\lambda))= \prod_{i=1}^{\ell(\lambda)} \frac{z+\big(\lambda_i-\frac{1}{2}\big)}{z-\big(\lambda_i+\frac{1}{2}\big)}
= \Phi\big(z- \tfrac{1}{2};\lambda\big).
\end{gather}
Taking the logarithm of the right hand side and expanding it at $z=\infty$,
\begin{gather*}
\log \Phi\big(z-\tfrac{1}{2}; \lambda\big)= \sum_{k=0}^\infty \frac{2p_{2k+1}(\lambda)}{2k+1} \big(z-\tfrac{1}{2}\big)^{-(2k+1)} \\
\hphantom{\log \Phi\big(z-\tfrac{1}{2}; \lambda\big)}{} = \sum_{k=0}^\infty
\frac{2p_{2k+1}(\lambda)}{2k+1} \frac{1}{z^{2k+1}}
\sum_{j=0}^\infty \frac{(2k+1)(2k+2) \cdots(2k+j)}{j!} \frac{1}{(2z)^j}.
\end{gather*}
Changing variable $j \mapsto n=2k+j+1$, we have
\begin{gather*}
\log \Phi\big(z-\tfrac{1}{2}; \lambda\big)= \sum_{n=1}^\infty \frac{1}{z^n}\sum_{k=0}^{\lfloor (n-1)/2 \rfloor}
2p_{2k+1}(\lambda) \frac{(2k+2) (2k+3)\cdots (n-1)}{(n-2k-1)!} \frac{1}{2^{n-2k-1}} \\
\hphantom{\log \Phi\big(z-\tfrac{1}{2}; \lambda\big)}{} = \sum_{n=1}^\infty \frac{1}{n z^n}
 \sum_{k=0}^{\lfloor (n-1)/2 \rfloor}
p_{2k+1}(\lambda)
\binom{n}{2k+1} \frac{1}{2^{n-2k-2}}.
\end{gather*}
On the other hand, from \eqref{eq:phi_Frobenius} we have
\begin{gather*}
\log \phi (z;D(\lambda)) = \sum_{n= 1}^\infty \frac{1}{n z^n} p_n^{\mathrm{super}}(D(\lambda)).
\end{gather*}
Comparing the coefficient of $z^{-n}$ in the above equations, we obtain the desired identity.
\end{proof}

\subsection{Moments of Rayleigh measures}

Recall the Rayleigh measure $\tau_{\lambda}$ defined in Section~\ref{subsec:Rayleigh_measures}.

\begin{Proposition} \label{prop:spin_Rayleigh} For a strict partition $\lambda$ and $k=1,2,3,\dots$, we have
\begin{gather*}
M_{2k}[\tau_{D(\lambda)}]= 4k p_{2k-1}(\lambda)+ \cdots \qquad \text{and} \qquad M_{2k+1}[\tau_{D(\lambda)}]=2k(2k+1) p_{2k-1}(\lambda)+ \cdots,
\end{gather*}
where dots are linear combinations of $\{p_{2j-1}(\lambda)\}_{j=1,2,\dots,k-1}$ with $\bQ_{\ge 0}$-coefficients. Therefore the family $\{M_{2k}[\tau_{D(\cdot)}]\}_{k=1,2,\dots}$ is an algebraic basis of $\Gamma$.
\end{Proposition}

\begin{proof}From Lemma \ref{lem:ordinary_Rayleigh_super} we have $M_{n}[\tau_{D(\lambda)}]=n p^{\mathrm{super}}_{n-1}(D(\lambda))+ \cdots$, where dots are a linear combination of $p^{\mathrm{super}}_{j}(D(\lambda))$ $(j=1,2,\dots,n-2)$, with $\bQ_{ \ge 0}$-coefficients. The desired claim follows from Proposition~\ref{prop:strict_super_p} immediately.
\end{proof}

\subsection{Moments and free cumulants of transition measures}

Recall the transition measure $\mathfrak{m}_\lambda$ defined in Section~\ref{subsec:transition_measures}.

\begin{Proposition}For strict partition $\lambda$ and $k=1,2,3,\dots$, we have
\begin{gather*}
M_{2k}[\mathfrak{m}_{D(\lambda)}] =2 p_{2k-1}(\lambda)+ \cdots \qquad \text{and} \qquad M_{2k+1}[\mathfrak{m}_{D(\lambda)}] =2k p_{2k-1}(\lambda)+ \cdots,
\end{gather*}
where dots are functions in $\Gamma$ of degree at most $2k-2$ with $\bQ_{\ge 0}$-coefficients. Therefore the family $\{M_{2k}[\mathfrak{m}_{D(\cdot)}]\}_{k=1,2,\dots}$ is an algebraic basis of $\Gamma$.
\end{Proposition}

\begin{proof}
Lemma \ref{lem:transition_Rayleigh} implies that
\begin{gather} \label{eq:strict_transition_expansion}
M_{n}[\mathfrak{m}_{D(\lambda)}]= \frac{1}{n} M_{n}[\tau_{D(\lambda)}] +
\sum_{\begin{subarray}{c} \nu =(\nu_1,\nu_2,\dots)\vdash n \\
\nu_1 < n \end{subarray}} z_\nu^{-1} M_{\nu_1}[\tau_{D(\lambda)}] M_{\nu_2}[\tau_{D(\lambda)}] \cdots.
\end{gather}
We here apply Proposition \ref{prop:spin_Rayleigh}. Let $n=2k+1$. Then $\frac{1}{n}M_{n}[\tau_{D(\lambda)}]= 2kp_{2k-1}(\lambda)+ f_k(\lambda)$, where $f_k \!\in\! \Gamma$ is of degree at most $2k{-}2$. We will estimate the degree of each term $M_{\nu_1}[\tau_{D(\lambda)}] M_{\nu_2}[\tau_{D(\lambda)}] {\cdots}\!$ in~\eqref{eq:strict_transition_expansion}. A~partition $\nu \vdash 2k+1$ with $\nu_1 <2k+1$ has at least one odd part $\nu_s$ with $\deg M_{\nu_s}[\tau_{D(\lambda)}] =\nu_s-2$ and has another part $\nu_t$ $(t \not=s)$
with $\deg M_{\nu_t}[\tau_{D(\lambda)}] \le \nu_t-1$. Hence the degree of each term $M_{\nu_1}[\tau_{D(\lambda)}]
M_{\nu_2}[\tau_{D(\lambda)}] \cdots$ is at most $(2k+1)-3=2k-2$. The case with $n=2k$ is similar (and easier).
\end{proof}

Now using the free cumulant-moment formula \eqref{eq:CM_formula}, we obtain the following result for free cumulants $R_n(D(\lambda))=R_n[\mathfrak{m}_{D(\lambda)}]$.

\begin{Corollary} \label{cor:free_cumulant}
For strict partition $\lambda$ and $k=1,2,3,\dots$, we have
\begin{gather*}
R_{2k}(D(\lambda)) =2 p_{2k-1}(\lambda)+ \cdots \qquad \text{and} \qquad R_{2k+1}(D(\lambda)) =2k p_{2k-1}(\lambda)+ \cdots,
\end{gather*}
where dots are functions in $\Gamma$ of degree at most $2k-2$ $($with $\bQ$-coefficients$)$. Therefore the family $\{R_{2k}(D(\cdot))\}_{k=1,2,\dots}$ is an algebraic basis of $\Gamma$.
\end{Corollary}

\subsection{Proof of main theorem}\label{subsec:spin_Kerov}

For each $k=1,2,3,\dots$ and $\lambda \in \mathcal{SP}$, we set
\begin{gather*}
\mathfrak{R}_{2k}(\lambda)= \frac{1}{2} R_{2k} (D(\lambda)).
\end{gather*}
The degree of $\mathfrak{R}_{2k}$ in $\Gamma$ is $2k-1$ as we proved in Corollary~\ref{cor:free_cumulant}.

\begin{proof}[Proof of Theorem~\ref{thm:MainThm}] As we saw in~\eqref{eq:spinCh_p1}, the difference $\mathfrak{p}_{2k-1}-p_{2k-1}$ belongs to $\mathbb{Q}[p_1,p_3,$ $\dots,p_{2k-3}]$ and has degree at most $2k-2$.
So does the difference $\mathfrak{R}_{2k}-p_{2k-1}$ by Corollary \ref{cor:free_cumulant}. Therefore
\begin{gather*}
\mathfrak{p}_{2k-1} - \mathfrak{R}_{2k} =(\mathfrak{p}_{2k-1}-p_{2k-1}) -(\mathfrak{R}_{2k} -p_{2k-1})
\end{gather*}
belongs to $\mathbb{Q}[p_1,p_3,\dots,p_{2k-3}]$ and is of degree at most $2k-2$. Again, using Corollary~\ref{cor:free_cumulant}, we see that $\mathbb{Q}[p_1,p_3,\dots,p_{2k-3}] =\mathbb{Q}[\mathfrak{R}_2,
\mathfrak{R}_4,\dots, \mathfrak{R}_{2k-2}]$. Hence $\mathfrak{p}_{2k-1} - \mathfrak{R}_{2k}$ is a polynomial in
$\mathfrak{R}_2, \mathfrak{R}_4,\dots, \mathfrak{R}_{2k-2}$ of degree at most $2k-2$.
\end{proof}

\subsection{Computations for spin Kerov polynomials}

In this subsection, we explain how to obtain formulas in Example~\ref{ex:SpinKerov}. We use Okounkov's idea which is employed in~\cite{Biane2003} for ordinary Kerov polynomials. Recall the function $G_{\mathfrak{m}_\lambda}(z)$ defined in Section~\ref{subsec:transition_measures}, which is the Cauchy transform of the transition measure associated with~$\lambda$.

\begin{Proposition} \label{prop:G_to_Phi}For each strict partition $\lambda$, we have
\begin{gather*}
G_{\mathfrak{m}_{D(\lambda)}}(z) = \frac{\Phi(z-1;\lambda)}{z\Phi(z;\lambda)}.
\end{gather*}
\end{Proposition}

\begin{proof}Immediate from Lemma~\ref{lem:G_to_phi} and the identity~\eqref{eq:phi_to_Phi}.
\end{proof}

\begin{Proposition}For a strict partition $\lambda$ and each $k=1,2,\dots$, we have
\begin{gather} \label{eq:FC_to_Phi}
\mathfrak{R}_{2k}(\lambda) = -\frac{1}{2(2k-1)} \big[z^{-2k}\big]\left(\frac{\Phi(z;\lambda)}{\Phi(z-1;\lambda)} \right)^{2k-1}.
\end{gather}
\end{Proposition}

\begin{proof}Using \eqref{eq:Biane_FC} we have
\begin{gather*}
\mathfrak{R}_{2k}(\lambda)= \frac{1}{2} R_{2k}(D(\lambda))= -\frac{1}{2(2k-1)} \big[z^{-1}\big] G_{\mathfrak{m}_{D(\lambda)}} (z)^{-(2k-1)}.
\end{gather*}
The desired expression follows immediately from Proposition~\ref{prop:G_to_Phi}.
\end{proof}

The expression \eqref{eq:FC_to_Phi} can be rewritten by~\eqref{eq:Phi_definition2} as
\begin{gather*}
\mathfrak{R}_{2k}(\lambda) = -\frac{1}{2(2k-1)} \big[z^{-2k}\big] \left\{
 \exp \left((2k-1) \sum_{j=1}^\infty \frac{2p_{2j-1}(\lambda)}{2j-1} z^{-(2j-1)}\right) \right.\\
\left.\hphantom{\mathfrak{R}_{2k}(\lambda) =}{} \times \exp \left(-(2k-1) \sum_{j=1}^\infty \frac{2p_{2j-1}(\lambda)}{2j-1} (z-1)^{-(2j-1)}\right) \right\}.
\end{gather*}
Using this with computer, we can obtain expansions of $\mathfrak{R}_{2k}$ in terms of power-sums.

\begin{Example} \label{ex:SpinFreeCummulants}
\begin{gather*}
\mathfrak{R}_{2} = p_1, \\
\mathfrak{R}_4= p_3-3p_1^2+p_1, \\
\mathfrak{R}_6 = p_5-10p_3 p_1 + \frac{10}{3} p_3 + \frac{50}{3} p_1^3 -15 p_1^2 +p_1, \\
\mathfrak{R}_8 = p_7 -14 p_5 p_1 -7 p_3^2+7 p_5 +98 p_3 p_1^2 - \frac{266}{3} p_3 p_1 -\frac{343}{3} p_1^4+ 7 p_3 +196 p_1^3 -35 p_1^2 +p_1, \\
\mathfrak{R}_{10} = p_9 -18 p_7 p_1 -18 p_5 p_3 +12 p_7 +162 p_5 p_1^2 +162 p_3^2 p_1- 198 p_5 p_1 -96 p_3^2 -972 p_3 p_1^3 \\
\hphantom{\mathfrak{R}_{10} =}{} +\frac{126}{5} p_5 +1674 p_3 p_1^2 +\frac{4374}{5} p_1^5 -330 p_3 p_1 -2430 p_1^4 +12 p_3 +810 p_1^3 -63 p_1^2 +p_1.
\end{gather*}
\end{Example}

\begin{proof}[Proof of Example \ref{ex:SpinKerov}] It can be obtained by comparing Examples \ref{ex:SpinCharacter} and \ref{ex:SpinFreeCummulants}.
\end{proof}

\section{Another double diagrams} \label{sec:another_double}

In this section, we deal with another double diagram of a strict partition $\lambda$. It is considered in \cite{Ivanov2006} and De Stavola's thesis \cite{DarioThesis}. We denote it by $\mathtt{D}(\lambda)$ and call it the symmetrized double diagram. We do not give its explicit definition here, see \cite[Fig.~11]{DarioThesis}. The diagram $\mathtt{D}(\lambda)$ is not a Young diagram associated to any partition but we can extend notions in Section~\ref{sec:Preliminary} to such diagrams, see \cite[Chapter~2.2]{Hora}.

The diagram $\mathtt{D}(\lambda)$ is drawn in the Russian style and has Kerov's interlacing coordinate
\begin{gather*}
\mathtt{x}_{-m} <\mathtt{y}_{-m} < \mathtt{x}_{-1} <\mathtt{y}_{-1} < \mathtt{x}_{ 0} < \mathtt{y}_{1} <\mathtt{x}_{1} < \cdots< \mathtt{y}_{m} <\mathtt{x}_{m},
\end{gather*}
where $\mathtt{x}_{-i}= -\mathtt{x}_i$, $\mathtt{y}_{-i}= -\mathtt{y}_i$, and $\mathtt{x}_0=0$. Note that these numbers except $\mathtt{x}_0$ belong to $\mathbb{Z}+\frac{1}{2}$. The corresponding Kerov's transition measure $\mathfrak{m}_{\mathtt{D}(\lambda)}$ is characterized by its Cauchy transform
\begin{gather*}
G_{\mathfrak{m}_{\mathtt{D}(\lambda)}}(z) = \frac{\prod\limits_{i=1}^m (z-\mathtt{y}_{-i})(z-\mathtt{y}_{i})}
{(z-\mathtt{x}_0)\prod\limits_{i=1}^m (z-\mathtt{x}_{-i})(z-\mathtt{x}_i)}
= \frac{\prod\limits_{i=1}^m \big(z^2-\mathtt{y}_i^2\big)}{z \prod\limits_{i=1}^m \big(z^2-\mathtt{x}_i^2\big)}.
\end{gather*}

\begin{Lemma}[{\cite[Proposition 2.6]{Ivanov2006}}] For a strict partition $\lambda$, we have
\begin{gather*}
G_{\mathfrak{m}_{\mathtt{D}(\lambda)}}(z)= \frac{\Phi\big(z-\tfrac{1}{2};\lambda\big)}{z \Phi\big(z+\tfrac{1}{2};\lambda\big)}.
\end{gather*}
\end{Lemma}
Compare this lemma with Proposition \ref{prop:G_to_Phi}. The remaining discussion is the same with that in the previous section, so we write only results.

For each $k=1,2,3,\dots$ and $\lambda \in \mathcal{SP}$, we set
\begin{gather*}
\mathtt{R}_{k}(\lambda)=\frac{1}{2} R_{k} [\mathfrak{m}_{\mathtt{D}(\lambda)}].
\end{gather*}
Then it has the expression
\begin{gather*}
\mathtt{R}_k(\lambda)= -\frac{1}{2(k-1)} \big[z^{-k}\big] \left( \frac{\Phi(z+\tfrac{1}{2};\lambda)}{\Phi\big(z-\tfrac{1}{2};\lambda\big)}\right)^{k-1} \\
\hphantom{\mathtt{R}_k(\lambda)}{} = -\frac{1}{2(k-1)} \big[z^{-k}\big] \left\{
\exp \left((k-1) \sum_{j=1}^\infty \frac{2p_{2j-1}(\lambda)}{2j-1} \big(z+\tfrac{1}{2}\big)^{-(2j-1)}\right)\right. \\
\left. \hphantom{\mathtt{R}_k(\lambda)=}{} \times \exp \left(-(k-1) \sum_{j=1}^\infty \frac{2p_{2j-1}(\lambda)}{2j-1} \big(z-\tfrac{1}{2}\big)^{-(2j-1)}\right)\right\}.
\end{gather*}
It is easy to see that $\mathtt{R}_k=0$ for odd $k$. We remark that $R_k(D(\lambda))$ does not vanish even if $k$ is odd.

\begin{Example}$\mathtt{R}_{2} =p_1$, $\mathtt{R}_{4} = p_3-3p_1^2+\frac{1}{4}p_1$.
\end{Example}

\begin{Proposition} \label{prop:Kerov2} For each $k=1,2,3,\dots$, the spin character $\mathfrak{p}_{2k-1}$ is a polynomial in $\mathtt{R}_2, \mathtt{R}_4,$ $\dots, \mathtt{R}_{2k}$ of the form
\begin{gather*}
\mathfrak{p}_{2k-1}= \mathtt{R}_{2k} + (\text{a polynomial in $\mathtt{R}_2,\mathtt{R}_4,\dots, \mathtt{R}_{2k-2}$}).
\end{gather*}
\end{Proposition}

\begin{Example}$\mathfrak{p}_{1}= \mathtt{R}_2$, $\mathfrak{p}_{3}= \mathtt{R}_4 +\frac{7}{4} \mathtt{R}_2$.
\end{Example}

From this example, we find that coefficients of polynomials in Proposition~\ref{prop:Kerov2} are \emph{not} integers. This indicates that these polynomials are ineligible for ``spin Kerov polynomials''. We thus believe that the $\mathfrak{R}_{2k}$ defined in Section~\ref{sec:MainPart} are the most appropriate choice for spin Kerov polynomials.

\subsection*{Acknowledgements}

The author acknowledges useful discussions with Valentin F\'{e}ray and Dario De Stavola. The research was supported by JSPS KAKENHI Grant Number 17K05281.

\pdfbookmark[1]{References}{ref}
\LastPageEnding

\end{document}